\documentclass{amsart}
\usepackage[dvipdfmx]{graphicx}
\usepackage{amssymb}

\newtheorem{theorem}{Theorem}

\newtheorem{lemma}[theorem]{Lemma}
\newtheorem{prop}[theorem]{Proposition}

\theoremstyle{remark}

\newtheorem{remark}[theorem]{Remark}
\newtheorem{example}[theorem]{Example}
\newtheorem*{acknowledgement}{Acknowledgement}

\numberwithin{equation}{section}

\begin{document}

\title{Toric Fano varieties associated to finite simple graphs}

\author{Yusuke Suyama}
\address{Department of Mathematics, Graduate School of Science, Osaka City University,
3-3-138 Sugimoto, Sumiyoshi-ku, Osaka 558-8585 JAPAN}
\email{d15san0w03@st.osaka-cu.ac.jp}



\date{\today}


\begin{abstract}
We give a necessary and sufficient condition for the nonsingular projective toric variety
associated to a finite simple graph to be Fano or weak Fano in terms of the graph.
\end{abstract}

\maketitle

\section{Introduction}

A {\it toric variety} of complex dimension $n$
is a normal algebraic variety $X$ over $\mathbb{C}$
containing the algebraic torus $(\mathbb{C}^*)^n$ as an open dense subset,
such that the natural action of $(\mathbb{C}^*)^n$ on itself extends to an action on $X$.
The category of toric varieties is equivalent to the category of fans,
which are combinatorial objects.

A nonsingular projective algebraic variety is called {\it Fano} (resp.\ {\it weak Fano})
if its anticanonical divisor is ample (resp.\ nef and big).
The classification of toric Fano varieties is a fundamental problem
and many results are known.
In particular, {\O}bro \cite{Obr07} gave an algorithm
classifying all such varieties for any dimension.
Sato \cite{Sat02} classified toric weak Fano 3-folds
that are not Fano but are deformed to Fano,
which are called toric {\it weakened Fano} 3-folds.

There is a construction of nonsingular projective toric varieties from finite simple graphs,
that is, associated toric varieties of normal fans of graph associahedra \cite{Pos06}.
We give a necessary and sufficient condition for the nonsingular projective toric variety
associated to a finite simple graph to be Fano (resp.\ weak Fano) in terms of the graph,
see Theorem \ref{Fano} (resp.\ Theorem \ref{weak Fano}).
The proofs are done by using the fact
that the intersection number of the anticanonical divisor with a torus-invariant curve
can be expressed by the number of connected components
of a certain induced subgraph (see Proposition \ref{Zelevinsky} and Lemma \ref{a}),
and by using graph-theoretic arguments.

The structure of the paper is as follows.
In Section 2, we review the construction of a toric variety
from a finite simple graph and we prepare some propositions for our proofs.
In Section 3, we give a condition for the toric variety to be Fano or weak Fano.

\begin{acknowledgement}
The author was supported by JSPS KAKENHI Grant Number 15J01000.
The author wishes to thank his supervisor, Professor Mikiya Masuda,
for his continuing support,
and Professor Akihiro Higashitani for his useful comments.
\end{acknowledgement}

\section{Toric varieties associated to graphs}

We fix a notation. Let $G$ be a finite simple graph.
We denote by $V(G)$ and $E(G)$ its node set and edge set respectively.
For $I \subset V(G)$, we denote by $G|_I$ the induced subgraph.
The {\it graphical building set} $B(G)$ of $G$ is defined to be
$\{I \subset V(G) \mid G|_I$ is connected, $I \ne \emptyset\}$.

We review the construction of a nonsingular projective toric variety from
a finite simple graph $G$.
In this paper we construct a toric variety from $G$ directly
(without using the graph associahedron).
First, suppose that $G$ is connected.
A subset $N$ of $B(G)$ is called a {\it nested set} if the following conditions are satisfied:
\begin{enumerate}
\item If $I, J \in N$, then we have either $I \subset J$ or $J \subset I$
or $I \cap J=\emptyset$.
\item If $I, J \in N$ and $I \cap J=\emptyset$, then $I \cup J \notin B(G)$.
\item $V(G) \in N$.
\end{enumerate}

\begin{remark}
The above definition of nested sets is different from the one in \cite[Definition 7.3]{Pos06}.
However, the two definitions are equivalent
for the graphical building set of a finite simple connected graph \cite[8.4]{Pos06}.
\end{remark}

The set $\mathcal{N}(B(G))$ of all nested sets of $B(G)$
is called the {\it nested complex}.

Let $V(G)=\{1, \ldots, n+1\}$.
We denote by $e_1, \ldots, e_n$ the standard basis for $\mathbb{R}^n$.
We put $e_{n+1}=-e_1-\cdots-e_n$ and $e_I=\sum_{i \in I}e_i$ for $I \subset V(G)$.
For $N \in \mathcal{N}(B(G))$,
we denote by $\mathbb{R}_{\geq 0}N$ the cone $\sum_{I \in N}\mathbb{R}_{\geq 0}e_I$,
where $\mathbb{R}_{\geq 0}$ is the set of non-negative real numbers.
The dimension of $\mathbb{R}_{\geq 0}N$ is $|N|-1$
since $V(G) \in N$ and $e_{V(G)}=0$.
We define $\Delta(G)=\{\mathbb{R}_{\geq 0}N \mid N \in \mathcal{N}(B(G))\}$.
Note that $\Delta(G)$ and $\mathcal{N}(B(G))$ are isomorphic
as ordered (by inclusion) sets.
$\Delta(G)$ is a nonsingular fan in $\mathbb{R}^n$
and the associated toric variety $X(\Delta(G))$ of complex dimension $n$
is nonsingular and projective.
In fact, $\Delta(G)$ is the normal fan of the graph associahedron of $G$
(see, for example \cite[8.4]{Pos06}).

If a finite simple graph $G$ is disconnected,
then we define $X(\Delta(G))$ to be the product of toric varieties
associated to connected components of $G$.

For a nonsingular complete fan $\Delta$ in $\mathbb{R}^n$ and $0 \leq r \leq n$,
We denote by $\Delta(r)$ the set of $r$-dimensional cones of $\Delta$.
We define a map $a:\Delta(n-1) \rightarrow \mathbb{Z}$ as follows.
For $\tau \in \Delta(n-1)$,
we take primitive vectors $v_1, \ldots, v_{n-1}$ such that
$\tau=\mathbb{R}_{\geq 0}v_1+\cdots+\mathbb{R}_{\geq 0}v_{n-1}$.
There exist distinct primitive vectors $v, v' \in \mathbb{Z}^n$
and integers $a_1, \ldots, a_{n-1}$ such that
$\tau+\mathbb{R}_{\geq 0}v$ and $\tau+\mathbb{R}_{\geq 0}v'$ are in $\Delta(n)$
and $v+v'+a_1v_1+\cdots+a_{n-1}v_{n-1}=0$.
Then we define $a(\tau)=a_1+\cdots+a_{n-1}$.
Note that the intersection number $(-K_{X(\Delta)}.V(\tau))$ is $2+a(\tau)$,
where $V(\tau)$ is the subvariety of $X(\Delta)$ corresponding to $\tau$
(see, for example \cite{Oda88}).

\begin{prop}\label{criterion}
Let $X(\Delta)$ be a nonsingular projective toric variety of complex dimension $n$.
Then the following hold:
\begin{enumerate}
\item $X(\Delta)$ is Fano if and only if $a(\tau) \geq -1$ for every $\tau \in \Delta(n-1)$.
\item $X(\Delta)$ is weak Fano if and only if $a(\tau) \geq -2$ for every $\tau \in \Delta(n-1)$.
\end{enumerate}
\end{prop}

(1) follows from the fact that $X(\Delta)$ is Fano if and only if
the intersection number $(-K_{X(\Delta)}.V(\tau))=2+a(\tau)$ is positive
for every $\tau \in \Delta(n-1)$ \cite[Lemma 2.20]{Oda88}.
In the case of toric varieties, 
$X(\Delta)$ is weak Fano if and only if the anticanonical divisor $-K_{X(\Delta)}$ is nef
\cite[Proposition 6.17]{Sat00}.
Since $-K_{X(\Delta)}$ is nef if and only if
$(-K_{X(\Delta)}.V(\tau))=2+a(\tau)$ is non-negative
for every $\tau \in \Delta(n-1)$, we get (2).

\begin{prop}\label{product}
Let $X(\Delta)$ and $X(\Delta')$ be nonsingular projective toric varieties
of complex dimension $m$ and $n$, respectively.
Then $X(\Delta) \times X(\Delta')$ is Fano (resp.\ weak Fano)
if and only if $X(\Delta)$ and $X(\Delta')$ are Fano (resp.\ weak Fano).
\end{prop}

\begin{proof}
We have $X(\Delta) \times X(\Delta')=X(\Delta \times \Delta')$,
where $\Delta \times \Delta'=\{\sigma \times \sigma'
\mid \sigma \in \Delta, \sigma' \in \Delta'\}$,
and any $(m+n-1)$-dimensional cone in $\Delta \times \Delta'$ is of the form
$\tau \times \sigma'$ for some $\tau \in \Delta(m-1)$ and $\sigma' \in \Delta'(n)$,
or $\sigma \times \tau'$ for some $\sigma \in \Delta(m)$ and $\tau' \in \Delta'(n-1)$.
Hence the proposition holds from
$a(\tau \times \sigma')=a(\tau), a(\sigma \times \tau')=a(\tau')$
and Proposition \ref{criterion}.
\end{proof}

\begin{prop}\label{Zelevinsky}
Let $G$ be a finite simple connected graph with $V(G)=\{1, \ldots, n+1\}$
and let $N \in \mathcal{N}(B(G))$ with $|N|=n$.
Then the following hold:
\begin{enumerate}
\item There exists a pair $\{J, J'\} \subset B(G) \setminus N$ such that
$N \cup \{J\}, N \cup \{J'\} \in \mathcal{N}(B(G))$ and $J \cup J' \in N$
\cite[Corollary 7.5]{Zel06}.
\item If $G|_{I_1}, \ldots, G|_{I_m}$ are the connected components of $G|_{J \cap J'}$,
then we have $I_1, \ldots, I_m \in N$ and
$e_J+e_{J'}-e_{I_1}-\cdots-e_{I_m}-e_{J \cup J'}=0$
\cite[Proposition 4.5 and Corollary 7.6]{Zel06}.
\end{enumerate}
\end{prop}

The following lemma follows immediately from Proposition \ref{Zelevinsky}.

\begin{lemma}\label{a}
Let $G$ be a finite simple connected graph
and let $N \in \mathcal{N}(B(G))$ with $|N|=|V(G)|-1$.
Then we have
\begin{equation*}
a(\mathbb{R}_{\geq0}N)=\left\{\begin{array}{ll}
-m & (J \cup J'=V(G)), \\
-m-1 & (J \cup J' \subsetneq V(G)), \end{array}\right.
\end{equation*}
where $\{J, J'\} \subset B(G) \setminus N$ is the pair in Proposition \ref{Zelevinsky}
and $m$ is the number of connected components of $G|_{J \cap J'}$.
\end{lemma}

\section{Main results}

First we characterize finite simple graphs
whose associated toric varieties are Fano.

\begin{theorem}\label{Fano}
Let $G$ be a finite simple graph.
Then the associated nonsingular projective toric variety $X(\Delta(G))$ is Fano if and only if
each connected component of $G$ has at most three nodes.
\end{theorem}

\begin{proof}
By Proposition \ref{product},
it suffices to show that for a finite simple connected graph $G$,
the toric variety $X(\Delta(G))$ is Fano if and only if $|V(G)| \leq 3$.

Let $V(G)=\{1, \ldots, n+1\}$. If the toric variety $X(\Delta(G))$ is Fano,
then we have $|(\Delta(G))(1)| \leq 3n$ when $n$ is even,
and $|(\Delta(G))(1)| \leq 3n-1$ when $n$ is odd \cite{Cas06}.
On the other hand, the lower bound for $f$-vectors of graph associahedra
is achieved for the graph associahedron corresponding to the path graph \cite{BV11}.
In particular, we have
$|(\Delta(G))(1)| \geq |(\Delta(L_{n+1}))(1)|=|B(L_{n+1})|-1=\cfrac{(n+1)(n+2)}{2}-1$,
where $L_{n+1}$ is the path graph on $\{1, \ldots, n+1\}$.
Thus we have the inequalities $3n \geq \cfrac{(n+1)(n+2)}{2}-1$ when $n$ is even,
and $3n-1 \geq \cfrac{(n+1)(n+2)}{2}-1$ when $n$ is odd.
These hold only for $n \leq 2$, so $|V(G)| \leq 3$.

Conversely, if $|V(G)| \leq 3$, then $X(\Delta(G))$ must be one of the following:
\begin{enumerate}
\item $V(G)=\{1\}, E(G)=\emptyset$: a point, which is understood to be Fano.
\item $V(G)=\{1, 2\}, E(G)=\{\{1, 2\}\}$: $\mathbb{P}^1$.
\item $V(G)=\{1, 2, 3\}, E(G)=\{\{1, 2\}, \{2, 3\}\}$: $\mathbb{P}^2$
blown-up at two points.
\item $V(G)=\{1, 2, 3\}, E(G)=\{\{1, 2\}, \{1, 3\}, \{2, 3\}\}$: $\mathbb{P}^2$
blown-up at three points.
\end{enumerate}
Thus $X(\Delta(G))$ is Fano for every case. This completes the proof.
\end{proof}

We characterize graphs whose associated toric varieties are weak Fano.
We denote by $K$ the diamond graph, that is, the graph obtained
by removing an edge from the complete graph on four nodes.

\begin{figure}[htbp]
\begin{center}
\includegraphics[width=3cm]{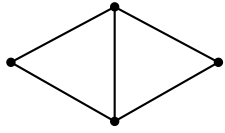}
\caption{the diamond graph $K$.}
\end{center}
\end{figure}

\begin{theorem}\label{weak Fano}
Let $G$ be a finite simple graph.
Then the associated nonsingular projective toric variety $X(\Delta(G))$ is weak Fano
if and only if for any connected component $G'$ of $G$
and for any proper subset $I$ of $V(G')$,
$G'|_I$ is neither a cycle graph of length at least four nor the diamond graph $K$.
\end{theorem}

\begin{example}
\begin{enumerate}
\item If $G$ is a cycle graph or $K$,
then the associated toric variety is weak Fano.
\item Toric varieties associated to trees and complete graphs are weak Fano.
\item The toric variety associated to the left graph in Figure \ref{ex} is weak Fano,
but the toric variety associated to the right graph is not weak Fano
because it has a cycle graph of length four as a proper induced subgraph.
\end{enumerate}
\end{example}

\begin{figure}[htbp]
\begin{center}
\includegraphics[width=2.5cm]{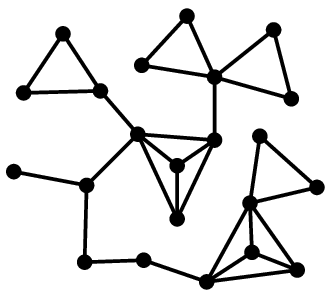}
\hspace{1.5cm}
\includegraphics[width=1.5cm]{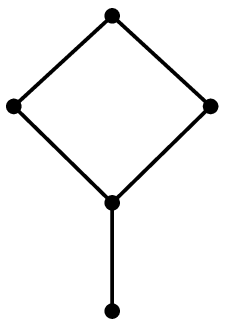}
\caption{examples.}
\label{ex}
\end{center}
\end{figure}

\begin{proof}[Proof of Theorem \ref{weak Fano}]
By Proposition \ref{product},
it suffices to show that for a finite simple connected graph $G$,
the toric variety $X(\Delta(G))$ is weak Fano if and only if
for any $I \subsetneq V(G)$, $G|_I$ is neither a cycle graph of length $\geq 4$ nor $K$.

First we show the necessity.
Suppose that there exists $I \subsetneq V(G)$
such that $G|_I$ is a cycle graph of length $l \geq 4$.
We may assume that
\begin{align*}
V(G)&=\{1, \ldots, n+1\}, n \geq l,\\
E(G|_{\{1, \ldots, l\}})&=\{\{1, 2\}, \{2, 3\}, \ldots, \{l-1, l\}, \{l, 1\}\},
\end{align*}
and $G|_{\{1, \ldots, k\}}$ is connected for every $1 \leq k \leq n+1$.
We consider the nested set
\begin{equation*}
N=\{\{1\}, \{1, 2\}, \ldots, \{1, \ldots, l-3\}, \{l-1\}, \{1, \ldots, l\},
\{1, \ldots, l+1\}, \ldots, \{1, \ldots, n+1\}\}.
\end{equation*}
The pair in Proposition \ref{Zelevinsky} is $J=\{1, \ldots, l-1\}$ and $J'=\{1, \ldots, l-3, l-1, l\}$.
Thus we have $J \cup J'=\{1, \ldots, l\} \subsetneq \{1, \ldots, n+1\}$
and $G|_{J \cap J'}=G|_{\{1, \ldots, l-3, l-1\}}$ has two connected components.
Hence we have $a(\mathbb{R}_{\geq 0}N)=-3$ by Lemma \ref{a}.
Therefore $X(\Delta(G))$ is not weak Fano by Proposition \ref{criterion}.

Suppose that there exists $I \subsetneq V(G)$ such that $G|_I$ is isomorphic to $K$.
We may assume that
\begin{align*}
V(G)&=\{1, \ldots, n+1\}, n \geq 4,\\
E(G|_{\{1, 2, 3, 4\}})&=\{\{1, 2\}, \{1, 3\}, \{1, 4\}, \{2, 3\}, \{2, 4\}\},
\end{align*}
and $G|_{\{1, \ldots, k\}}$ is connected for every $1 \leq k \leq n+1$.
We consider the nested set
\begin{equation*}
N=\{\{3\}, \{4\}, \{1, 2, 3, 4\}, \{1, 2, 3, 4, 5\}, \ldots, \{1, \ldots, n+1\}\}.
\end{equation*}
The pair in Proposition \ref{Zelevinsky} is $J=\{1, 3, 4\}$ and $J'=\{2, 3, 4\}$.
Thus we have $J \cup J'=\{1, 2, 3, 4\} \subsetneq \{1, \ldots, n+1\}$
and $G|_{J \cap J'}=G|_{\{3, 4\}}$ consists of two isolated nodes.
Hence we have $a(\mathbb{R}_{\geq 0}N)=-3$ by Lemma \ref{a}.
Therefore $X(\Delta(G))$ is not weak Fano by Proposition \ref{criterion}.

We prove the sufficiency.
Suppose that $X(\Delta(G))$ is not weak Fano.
By Proposition \ref{criterion}, there exists $N \in \mathcal{N}(B(G))$ such that
$|N|=|V(G)|-1$ and $a(\mathbb{R}_{\geq 0}N) \leq -3$.
We have the pair $\{J, J'\}$ in Proposition \ref{Zelevinsky}
and the number of connected components of $G|_{J \cap J'}$
is greater than or equal to two by Lemma \ref{a}.
Let $G|_{I_1}, \ldots, G|_{I_m}$ be the connected components of $G|_{J \cap J'}$.
We take $x \in I_1, x' \in I_2$
and simple paths $x=y_1, y_2, \ldots, y_r=x'$ in $G|_J$
and $x=z_1, z_2, \ldots, z_s=x'$ in $G|_{J'}$. Let
\begin{align*}
p&=\mathrm{max}\{1 \leq i \leq r \mid y_i \in I_1, 1 \leq \exists j \leq s: y_i=z_j\},\\
q&=\mathrm{min}\{p+1 \leq i \leq r \mid y_i \in (I_2 \cup \cdots \cup I_m) \setminus I_1,
1 \leq \exists j \leq s: y_i=z_j\}.
\end{align*}
Then we have two simple paths between $y_p$ and $y_q$.
The two paths have no common nodes except $y_p$ and $y_q$.
Since $y_p \in I_1$ and $y_q \in (I_2 \cup \cdots \cup I_m) \setminus I_1$,
we have $\{y_p, y_q\} \notin E(G)$ and the number of edges of
each path is greater than or equal to two.
Thus we obtain a simple cycle of length $\geq 4$ containing $y_p$ and $y_q$.
Hence we may assume that:
\begin{enumerate}
\item $V(G)=\{1, \ldots, n+1\}$.
\item There exists an integer $l$ such that $4 \leq l \leq n+1$
and $\{1, 2\}, \{2, 3\}, \ldots, \{l-1, l\}, \{l, 1\} \in E(G)$.
\item There exists an integer $k$ such that $3 \leq k \leq l-1$
and $\{1, k\} \notin E(G)$.
\end{enumerate}
Moreover, we may assume that $\{i, j\} \notin E(G)$ for every
\begin{itemize}
\item $1 \leq i<j \leq k$ where $j-i \geq 2$,
\item $k \leq i<j \leq l$ where $j-i \geq 2$,
\item $k+1 \leq i \leq l-1$ and $j=1$,
\end{itemize}
since if such an edge exists,
then we can replace the cycle by a shorter cycle containing the edge.

We find a cycle graph of length $\geq 4$ or $K$ as an induced graph of $G$.

{\it The case where $\{2, l\} \notin E(G)$}. We consider
\begin{align*}
i_\mathrm{min}&
=\mathrm{min}\{2 \leq i \leq k \mid k+1 \leq \exists j \leq l: \{i, j\} \in E(G)\},\\
j_\mathrm{max}&
=\mathrm{max}\{k+1 \leq j \leq l \mid \{i_\mathrm{min}, j\} \in E(G)\}.
\end{align*}
Then the induced subgraph by the subset
\begin{equation*}
\{1, 2, \ldots, i_\mathrm{min}, j_\mathrm{max}, j_\mathrm{max}+1, \ldots, l\} \subset V(G)
\end{equation*}
is a cycle graph of length $\geq 4$.

\begin{figure}[htbp]
\begin{center}
\includegraphics[width=6cm]{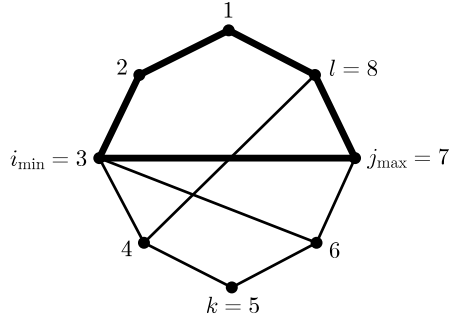}
\caption{a cycle graph as an induced subgraph.}
\end{center}
\end{figure}

{\it The case where $\{2, l\} \in E(G)$}.
If there exists an integer $j$ such that $k+1 \leq j \leq l-1$ and $\{2, j\} \in E(G)$,
then we have a cycle graph of length $\geq 4$ or $K$ as an induced subgraph.
If $\{2, j\} \notin E(G)$ for any $k+1 \leq j \leq l-1$,
then we consider
\begin{align*}
i_\mathrm{min}&
=\mathrm{min}\{3 \leq i \leq k \mid k+1 \leq \exists j \leq l: \{i, j\} \in E(G)\},\\
j_\mathrm{max}&
=\mathrm{max}\{k+1 \leq j \leq l \mid \{i_\mathrm{min}, j\} \in E(G)\}.
\end{align*}
The induced subgraph by the subset
\begin{equation*}
\{2, 3, \ldots, i_\mathrm{min}, j_\mathrm{max}, j_\mathrm{max}+1, \ldots, l\} \subset V(G)
\end{equation*}
is a cycle graph. If its length is at least four,
then we have a desired induced subgraph.
If its length is three, then we have an induced subgraph $K$
by combining the cycle graph with two edges $\{1, 2\}$ and $\{l, 1\}$.

Thus we obtain a cycle graph of length $\geq 4$ or $K$ as an induced subgraph of $G$.
If $G$ is a cycle graph of length $\geq 4$ or $K$,
it can be easily checked that for any $N \in \mathcal{N}(B(G))$ such that $|N|=|V(G)|-1$,
the number of connected components of $G|_{J \cap J'}$ is at most two.
Moreover, if the number of connected components is two,
then we must have $J \cup J'=V(G)$.
Hence we have $a(\mathbb{R}_{\geq0}N) \geq -2$ by Lemma \ref{a}.
So $X(\Delta(G))$ is weak Fano by Proposition \ref{criterion}, which is a contradiction.
Thus $G$ has a cycle graph of length $\geq 4$ or $K$ as a proper induced subgraph of $G$.
This completes the proof.
\end{proof}


\begin{thebibliography}{9}
\bibitem{BV11} V. Buchstaber and V. Volodin,
{\it Sharp upper and lower bounds for nestohedra},
Izv. Math. {\bf 75} (2011), no.\ 6, 1107--1133.
\bibitem{Cas06} C. Casagrande,
{\it The number of vertices of a Fano polytope},
Ann. Inst. Fourier {\bf 56} (2006), 121--130.
\bibitem{Obr07} M. {\O}bro,
{\it An algorithm for the classicifation of smooth Fano polytopes},
arXiv:0704.0049.
\bibitem{Oda88} T. Oda,
{\it Convex Bodies and Algebraic Geometry. An Introduction to the Theory of Toric Varieties},
Ergeb. Math. Grenzgeb. (3) {\bf 15}, Springer-Verlag, Berlin, 1988.
\bibitem{Pos06} A. Postnikov,
{\it Permutohedra, Associahedra, and Beyond},
Int. Math. Res. Not. {\bf 2009}, no.\ 6, 1026--1106.
\bibitem{Sat00} H. Sato,
{\it Toward the classification of higher-dimensional toric Fano varieties},
Tohoku Math. J. {\bf 52} (2000), 383--413.
\bibitem{Sat02} H. Sato,
{\it The classification of smooth toric weakened Fano 3-folds},
Manuscripta Math. {\bf 109} (2002), no.\ 1, 73--84.
\bibitem{Zel06} A. Zelevinsky,
{\it Nested Complexes and their Polyhedral Realizations},
Pure Appl. Math. Q. {\bf 2} (2006), no.\ 3, 655--671.
\end{thebibliography}
\end{document}